\documentclass{article}
\usepackage[utf8]{inputenc}

\title{A Bound for Higher Degree Automorphisms}
\author{Robert Rust}

\usepackage{graphicx}
\usepackage{amsmath}
\usepackage{amsthm}
\usepackage{amssymb}
\usepackage{wasysym}
\usepackage{xcolor}
\usepackage{enumitem}
\usepackage[mathscr]{euscript}
\usepackage{titlesec}
\usepackage{amsmath}
\usepackage{graphicx}
\usepackage{indentfirst}
\usepackage{accents}
\newlist{toclikeI}{enumerate}{1}

\newlist{toclikeII}{enumerate}{1}
\setlist[toclikeII]{align=left}

\newtheorem{theorem}{Theorem}[section]

\newtheorem{lemma}[theorem]{Lemma}
\newtheorem{definition}[theorem]{Definition}
\newtheorem*{fact}{Fact}
\newtheorem{conjecture}[theorem]{Conjecture}
\begin{document}

\begin{titlepage}
    \begin{center}
\vspace*{-.5in}
\includegraphics[width=4in]{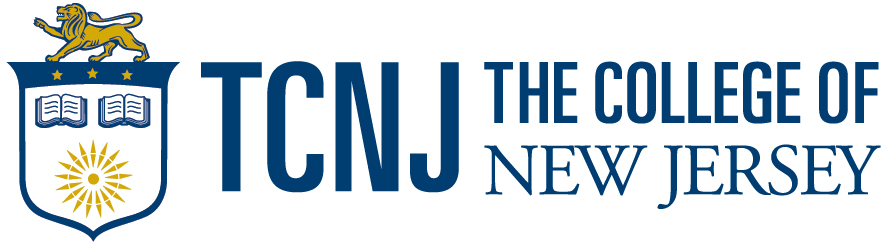}\\
\vspace*{1.5in}
\Huge{\bf A Bound for Higher Degree Automorphisms of $\mathbf{F_n}$}\\
\vspace*{1in}
\LARGE{Robert Rust}\\
\vspace*{1.5in}

\large{Honors Thesis\\
Department of Mathematics and Statistics\\
Advisor: Professor Andrew Clifford\\}
\vspace*{.5in}
{May 2023}\\
\end{center}
\end{titlepage}

\begin{abstract}
    The Whitehead Model of free groups can be used to measure the complexity, or degree, of automorphisms of free groups. The bound for the degree of the $f \circ g$ for deg$(f) =$ deg($g) = 0$ had previously been discovered. We extend this result to the case where at least one of our automorphisms has degree 0.
\end{abstract}
\newpage
\noindent

\tableofcontents
\newpage

\section{Free Groups}

 Let $X = \{a, b, c, \ldots\}$ be a set of elements. We call \[X \cup X^{-1} = \{a, a^{-1},b, b^{-1},c, c^{-1}, \ldots\}\] the \textbf{balanced alphabet} on X, with the elements of our alphabet called \textbf{letters} (we assume $X \cap X^{-1} = \emptyset$). We call a finite string of letters $w \in W[X] = \{\alpha_1 \alpha_2 \ldots \alpha_n | \alpha_i \in X \cup X^{-1}\}$ a \textbf{word}. The \textbf{length} of a word, denoted as $|w|$ is the number of letters that occur in the word. The \textbf{empty word} is a word which has length zero, henceforth written as just 1 \cite{Stillwell1995}.\\
 When we have a string of the letters of the form $x_jx_j^{-1}$, we call this a \textbf{trivial relator} and we say a word $w$ is \textbf{freely reduced } if no trivial relators occur in $w$. We call two words $w$ and $w'$ \textbf{freely equivalent} if $w'$ can be obtained from $w$ by inserting or removing a finite number of trivial relators. We denote this $w \approx w'$.\\
 We can define an operation on the words of $X \cup X^{-1}$ that is both associative and has the identity 1, making it a monoid. Taking the product of two words $w_1$ and $w_2$, we write $w_1w_2$ as the juxtaposition of the two words. For example, for the words $w_1 = ac^{-1}d$ and $w_2 = b^{-1}e$, their product is $w_1w_2 = ac^{-1}db^{-1}e$ \cite{Lyndon2001}.\\

 \begin{fact}
 Let $S$ be a set of distinct symbols. Then every word in W[S] has a unique freely reduced form. \cite{Hatcher2001}
 \end{fact}

We use this to claim that freely equivalent words share reduced forms. We will denote that equivalence class of a word $w$ with $[w]$. As well, we can say that if $w \approx \tau$, then $w \lambda \approx \tau \lambda$ for any words $w, \tau,$ and $\lambda$.

It is also true that $W[X]$ satisfies the group axiom of the existence of inverses. For example, the inverse of $w = ba^{-1}c$ is $w^{-1} = c^{-1}ab^{-1}$. We see that $ww^{-1} = ba^{-1}cc^{-1}ab^{-1} \approx 1.$ In general, inverses can be found by reversing the letters and taking the respective inverses of each. We note that $(x^{-1})^{-1}$ is the same as $x$. We can use this to define the following:

\begin{definition}
    The set $F[X] = \{ [w] | w \in W[X]\}$ with the operation defined by $[u][v] = [uv]$ is a group and is called a free group with basis $X$.
\end{definition}

\begin{definition}
    Let $X$ be a subset of a group $F$. Then \textbf{$F$ is a free group with basis $X$} provided the following holds: if $\phi$ is any function from the set $X$ into a group $H$, then there exists a unique extension of $\phi$ to a homomorphism $\phi^*$ from $F$ into $H$ satisfing $\phi^*(x) = \phi(x) \forall x \in X$. \cite{Lyndon2001}
\end{definition}

\begin{fact}
    Any two bases have the same cardinality. So if $B \subseteq F_n$ so that $B$ generates $F_n$ and $|B| = n$, $B$ is a basis.
\end{fact}

An example of a set that forms a basis for $F_3$ is the set $X = \{w_1, w_2, w_3\} = \{a, a^{-1}bc, c^{-1}a^{-1}\}.$ Clearly, we can take just $w_1$ and get $a$. If take $w_1^{-1}w_3^{-1}$, we get $a^{-1}ac$, which reduces down to just $c$. Finally, now that we have out $a$ and $c$ elements, we can take $aw_2c^{-1}$ to get $aa^{-1}bcc^{-1}$, or just $b$. Thus, we can get to the set $\{a,b,c\}$, which are our 3 basis elements of $F_3$.

\section{The Whitehead Model}

We will use $M_n = \#^n(S^1 \times S^2)$ to study elements and automorphisms of $F_n$. The model that we will be using was developed by J.H.C. Whitehead in 1936 but is still a relevant way to explore free groups \cite{10.2307/1968618}.

The Seifert-van Kampen theorem, often called just van Kampen’s
theorem, is a result that relates the fundamental group of a space $\mathbf{X}$ that is the union of path-connected subspaces with the fundamental groups of its constituent subspaces, i.e. the subspaces $A$ and $B$ that comprise that union. The fundamental group of a topological space $\mathbf{X}$, denoted $\pi_1(\mathbf{X},x_0)$, is homotopy classes of closed curves with basepoint $x_0$ in $\mathbf{X}$. If $A$ and $B$ are 3-manifolds then $\pi_1(A\#B) = \pi_1(A) * \pi_1(B)$ \cite{Hatcher2001}. If the intersection
of these subspaces is also path-connected then we will be able write $\pi_1(\mathbf{X})$ in terms of the free product of $\pi_1(A)$ and $\pi_1(B)$ with respect to a certain normal
subgroup. For path-connected subspaces $A$ and $B$ of $\mathbf{X}$ whose union is the entirety of $X$ consider the set of points $A \cap B$. The homomorphisms $i_* : \pi_1(A \cap B) \to \pi_1(A)$
and $j_* : \pi_1(A \cap B) \to \pi_1(B)$ are induced by the inclusion maps $A \cap B \overset{i}\hookrightarrow A$ and $A \cap B \overset{j}\hookrightarrow B$, respectively \cite{Hatcher2001}.

\begin{theorem}{\textbf{van-Kampen's Theorem}}
    If a topological space $\mathbf{X}$ is the union of path-connected, ``closed" sets $A$ and $B$ then if $A \cap B$ is path-connected,
    \[ \pi_1(\mathbf{X}) \cong \frac{\pi_1(A)*\pi_1(B)}{\langle \langle i_*(\omega)j_*(\omega)^{-1}|\omega \in \pi_1(A \cap B)\rangle\rangle} \]
    where $\langle \langle i_*(\omega)j_*(\omega)^{-1}|\omega \in \pi_1(A \cap B)\rangle\rangle$ is the smallest normal subgroup containing these elements. \cite{Hatcher2001}
\end{theorem} 
\bigskip

We use the following definition of a collar (or bicollar) to help us define the embedding of our spheres in $M_n$:
\begin{definition}\cite{Rolfsen1976}
    A subset $X \subset Y$ is said to be \textbf{bicollared} (in $Y$) if there exists an embedding $b: X \times [-1,1] \to Y$ such that $b(x,0) = x$ when $x \in X$. The map $b$, or its image, is the \textbf{bicollar} itself.
\end{definition}

We define a \textbf{Collared Sphere Basis} in $M_n$ as an $n$-tuple of 2 spheres along with a labeling and orientation that are mutually disjoint and whose complement is connected. The ``negative" side of our collaring is the negative labeling of our sphere and the positive collar will be the positive labeling of our spheres. We denote this $(S_1, S_2, \ldots, S_n)$ so that \begin{enumerate}
    \item Each $S_i$ is homeomorphic to a 2-sphere.
    \item $S_i \cap S_j = \emptyset$ for $i \neq j$.
    \item Each $S_i$ has a bicollar $B_i = S^2 \times [-1,1]$ and all $B_i$ are disjoint and do not disconnect.
    \item $M_n - \bigcup\limits_{i=1}^{n}S_i$ is connected.
\end{enumerate}

We can then formulate our \textbf{standard sphere basis} as the $n$ spheres that are identified by the boundaries of the removed 3-balls in the construction $M_n = \#^n(S^1 \times S^2)$, henceforth denoted $\mathfrak{G}$.\\

\section{Diagrams}
When we draw a diagram, we will draw the $n$ standard spheres as $2n$ spheres where each pair is identified with each other. This is equivalent to being bicollared because we are taking each $S_i$ with its bicollar and splitting it into its negative side and positive side.
\begin{center}
    \includegraphics[width=8.0cm]{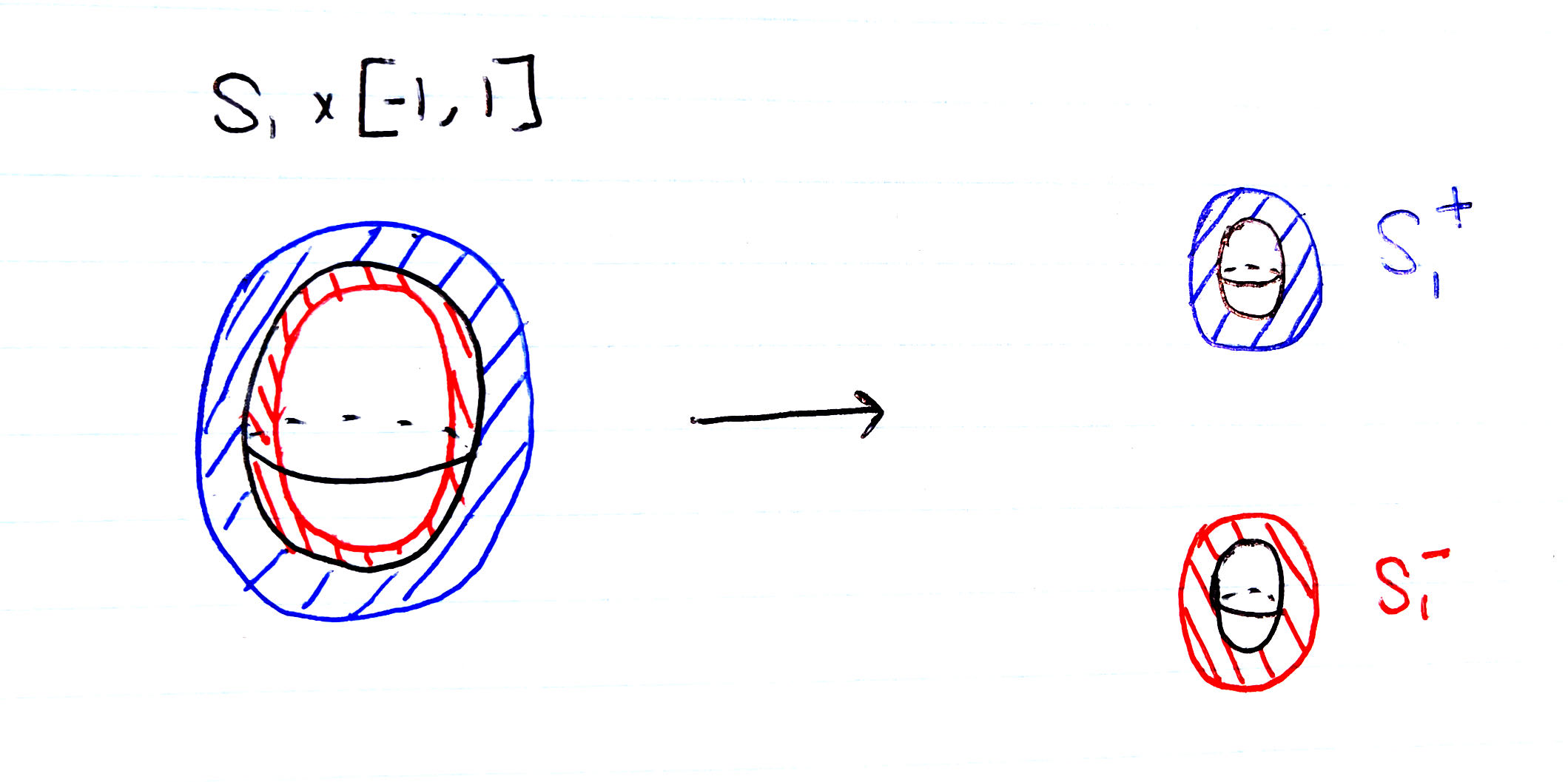}\\
    \end{center}
    We see an example here of a Sphere $S_1$ and its bicollar $B_1$, with the red reprsenting $S_1 \times [-1,0]$ and blue representing $S_1 \times (0,1]$. We can then draw our picture of $S_1$ with each label corresponding to the negative and positive side of the bicollar. 
The following helps us to understand why the Whitehead model is used to model free group automorphisms.
\begin{enumerate}
\item A self-homeomorphism of $M_n$ does two things:
\begin{enumerate}
    \item Takes the Standard Sphere Basis to another Collared Sphere basis
    \item Induces an automorphism of $F_n$ using the $\pi_1$ functor.
\end{enumerate}
\item Moreover, given any sphere basis, there is a self-automorphism of $M_n$ taking the standard sphere basis to it (preserving orientation and labeling).
\item Now we know that given any automorphism of $F_n$, we can draw a Collared Sphere Basis that induces this automorphism when we trace a path as defined later in this section.
\end{enumerate}

We will now define the space in which we call $\widetilde{M}_n$ as  $S^3 - (\bigcup\limits_{i=1}^{n}\mathring{B}_i^-) - (\bigcup\limits_{i=1}^{n}\mathring{B}_i^+)$. There exists an identification map $\widetilde{M}_n \to M_n$ and we will use $\widetilde{M}_n$ as we define our automorphisms of $F_n$.

Given a sphere basis, to find the the representing automorphism, we set a fixed point $x^*$ (usually chosen away from the spheres) and draw the closed, oriented curves that pass through the standard sphere basis for each pair of sphere $a, b, \ldots, n$ ($n$ comes from $F_n$, not the alphabet letter $n$). The corresponding spheres (with orientation) that we pass through to achieve this gives us our mapping. Without loss of generality, we always draw the standard sphere basis from left to right with the positive side above the negative side. Two collared sphere bases which are equivalent under isotopies induce the same automorphism.

We see in the example below that in order to pass through the standard sphere $a$, we must pass through the red $a$ sphere in the positive direction. Thus $a \to a$. To pass through the standard sphere $b$, we enter the red $b$ sphere in the positive direction, then exit the red $a$ sphere in the positive direction. Thus $b \to ab$ and our automorphism of $F_2$ is fully defined.

\begin{center}
    \includegraphics[width=8.0cm]{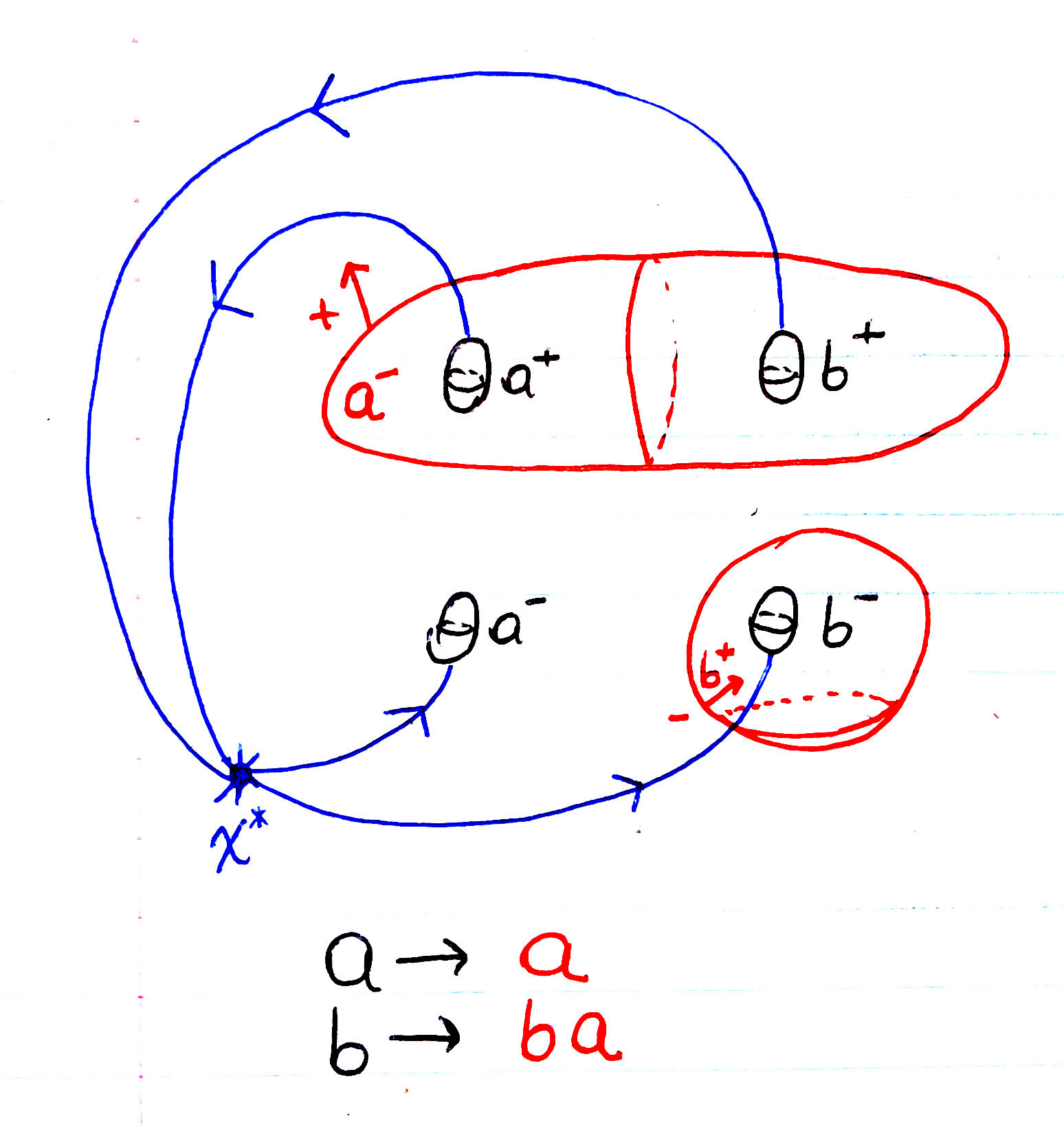}\\
    Figure 1: An automorphism of $F_2$.
\end{center}

It is worth noting that these homeomorphisms of $M_n$ are not unique to an automorphism. By way of isotopy, we can move the spheres around and define the same automorphism but with a different amount of intersections with the standard sphere basis. All of the submanifolds will remain properly embedded and all intersections are transverse. 

\section{Degrees}

We will define now what we mean when we say the \textbf{degree} of spheres and automorphisms. We need to discuss what we mean by \textit{components} of spheres to count intersections and define our degree. We will denote our nonstandard spheres as $\mathfrak{R}$ and $\mathfrak{T}$.

\begin{definition}
    In $\widetilde{M}_n$, given a sphere $A \in \mathfrak{R}$, a \textbf{component} of $A$ is either $S^2$, an \textbf{end-cap}, which intersects the standard sphere basis once, or a \textbf{tunnel} (not necessarily a cylinder), which intersects the standard sphere basis more than once.
\end{definition}

\begin{center}
    \includegraphics[width=9.0cm]{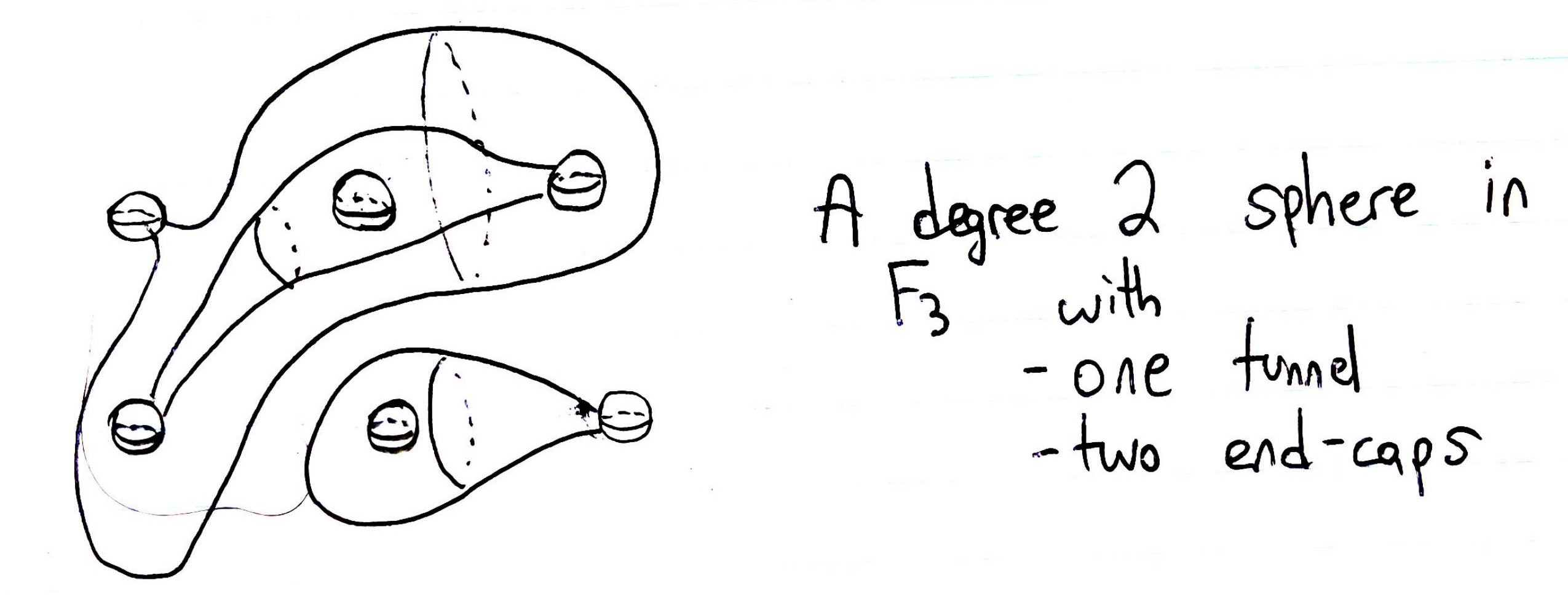}\\
    Figure: An example of the fact above with a degree 2 sphere and its 3 components.
\end{center}

\begin{fact}{\label{Sphere}}
    A sphere $A$ in $\widetilde{M}_m$ with degree 0 is just $S^2$.
\end{fact}

\begin{fact}
        A sphere $A$ in $\widetilde{M}_m$ with degree 1 has 2 end-caps.
\end{fact}

\begin{fact}
    A sphere $A$ in $\widetilde{M}_m$ with degree 2 has 2 end-caps and 1 tunnel.
\end{fact}

 Let $S_1$ and $S_2$ be spheres in $M_n$. Define $\#(S_1, S_2)$ as the number of components of $S_1 \cap S_2$. Given a sphere $A \in \mathfrak{R}$, define deg$(A) = \sum_{i=1}^n\#(A, S_i)$ where $S_i$ are the spheres of our standard sphere basis.\\
 Now, let $\mathscr{S}_1$ and $\mathscr{S}_2$ be sphere bases. Define \[\#(\mathscr{S}_1, \mathscr{S}_2) = \sum_{S_i \in \mathscr{S}_1} \sum_{S_j \in \mathscr{S}_2} 
 \#(S_i, S_j).\]

 Now we can say for an automorphism $f$, \[\text{deg}(f) = \text{min}\{\#(\mathscr{S}_0, \mathscr{S}_1) | \mathscr{S}_1 \ \text{represents} \ f\}.\]

\section{Composition of Automorphisms}

As a reminder, all submanifolds are properly embedded and all intersections are transverse. We will name our automorphisms $f_{A,B}$ where A is the input Collared Sphere Basis and B is the output Collared Sphere Basis. In order to compose two automorphisms $f$ and $g$ where $f = f_{\mathfrak{G},\mathfrak{T}}$ and $g = g_{\mathfrak{G},\mathfrak{R}}$, we must take the $g^{-1} = g_{\mathfrak{R},\mathfrak{G}}$. This allows us to find a new automorphism $f \circ g = h_{\mathfrak{R},\mathfrak{T}}.$

To find the inverse of an automoprhism, $g$ for example, we use the same base point as we used when we traced the automorphism. Now, instead of passing in the positive direction through the standard sphere basis and writing down which spheres we pass through, we will almost do the opposite. Now, we will trace our closed curves that pass through each sphere in the positive direction and write down how they pass through the standard spheres. We must note that we must only cross through a nonstandard sphere once and must not pass through any other nonstandard sphere. When we draw a new picture, we take the new input sphere basis and make this our new standard sphere basis. Our output spheres can be drawn with our new stanard sphere basis and the corresponding automorphism will remain the same. An example of finding an inverse and using it to compose can be see below:

\begin{center}
    \includegraphics[width=9.0cm]{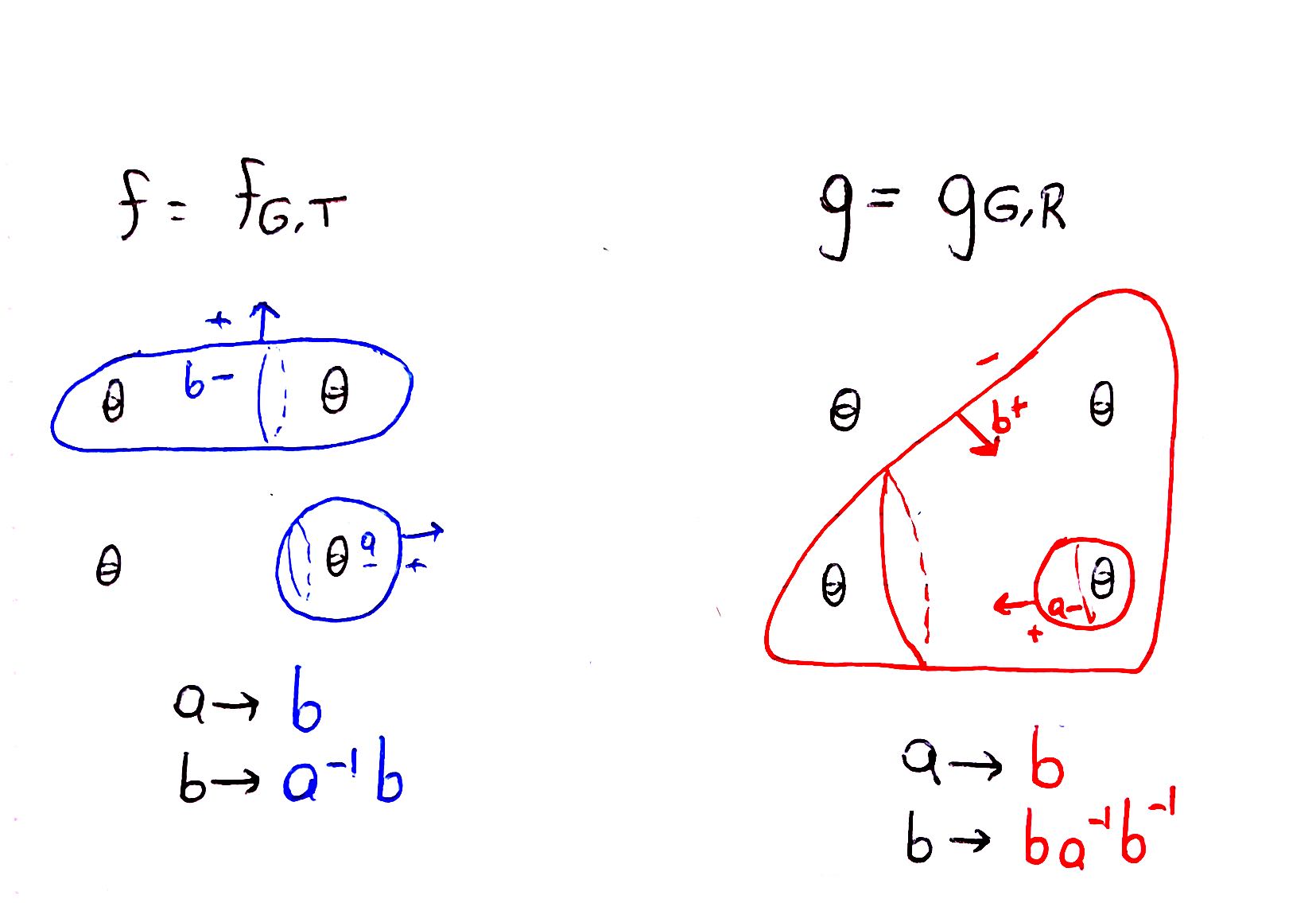}\\
        Figure 2: Two Automorphisms of $F_2$ we are going to compose.\\
\end{center}
    \includegraphics[width=9.0cm]{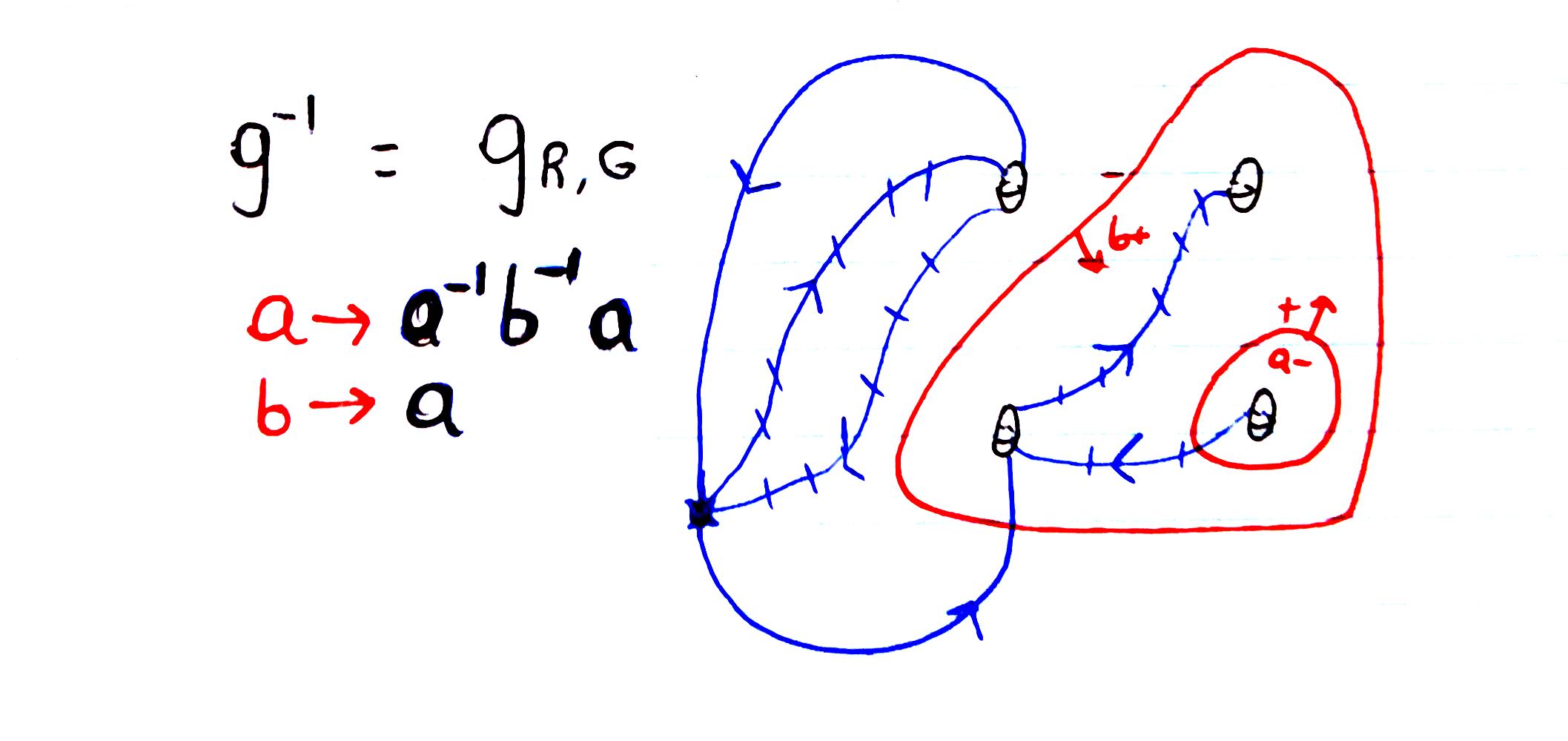}\\
     In order to compose $f$ and $g$, we must take the inverse of $g$. To find this automorphism, we pass through the $\mathfrak{R}$ spheres with positive orientation and trace which standard spheres we passed through. Our slashed-blue path passes through the red $a$ sphere in the positive direction and passes through the standard spheres in the order of $a^{-1}b^{-1}a$. Likewise, our smooth-blue path crosses through the red $b$ sphere and only crosses though the standard $a$ sphere in the positive direction. \\
\begin{center}
     \includegraphics[width=9.0cm]{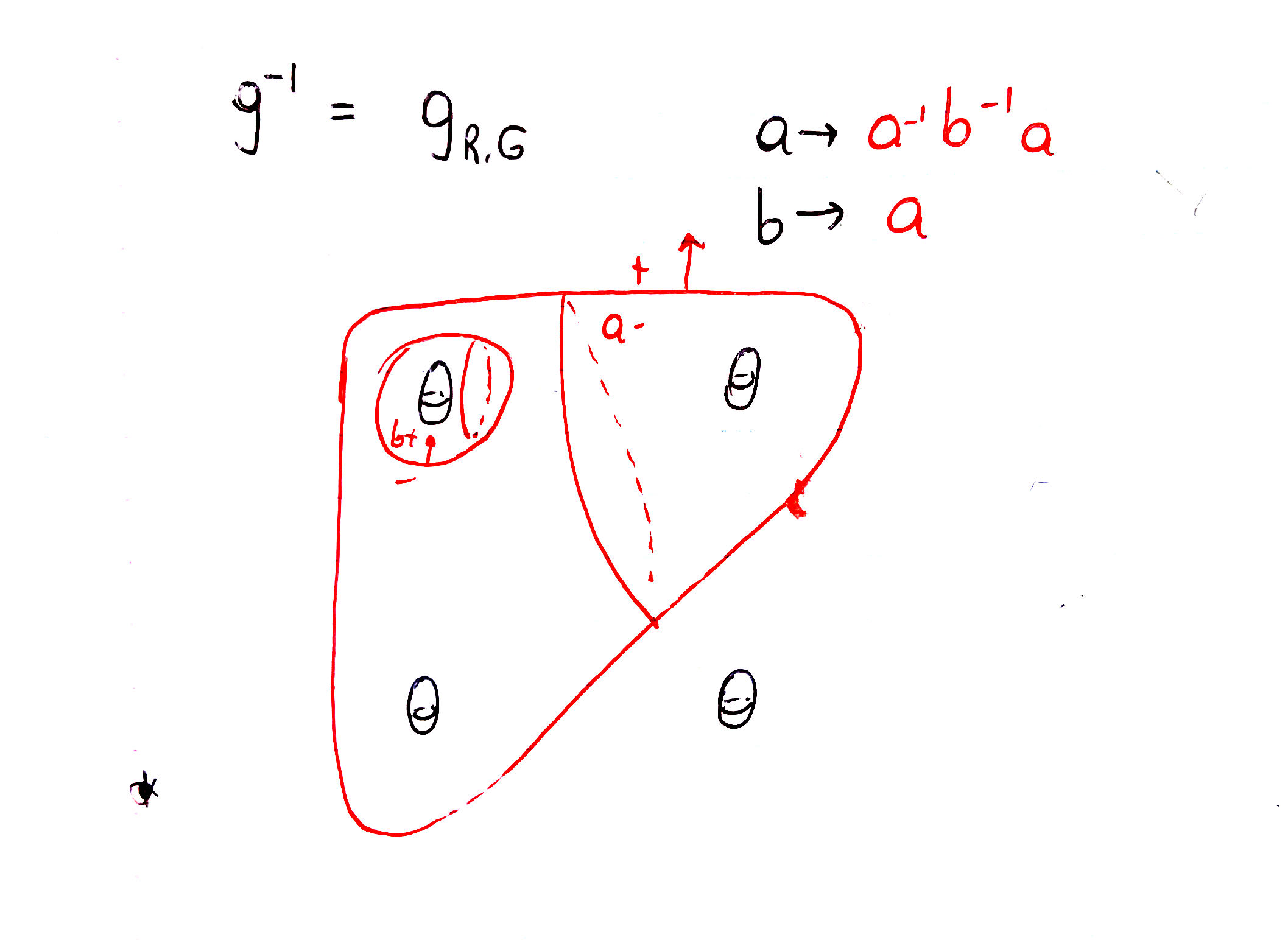}\\
\end{center}
     We can redraw a homeomorphism that describes the new automorphism for $g^{-1}$ that we will use for our composition. Notice that $\mathfrak{R}$ has become our new standard sphere basis.\\
\begin{center}
      \includegraphics[width=9.0cm]{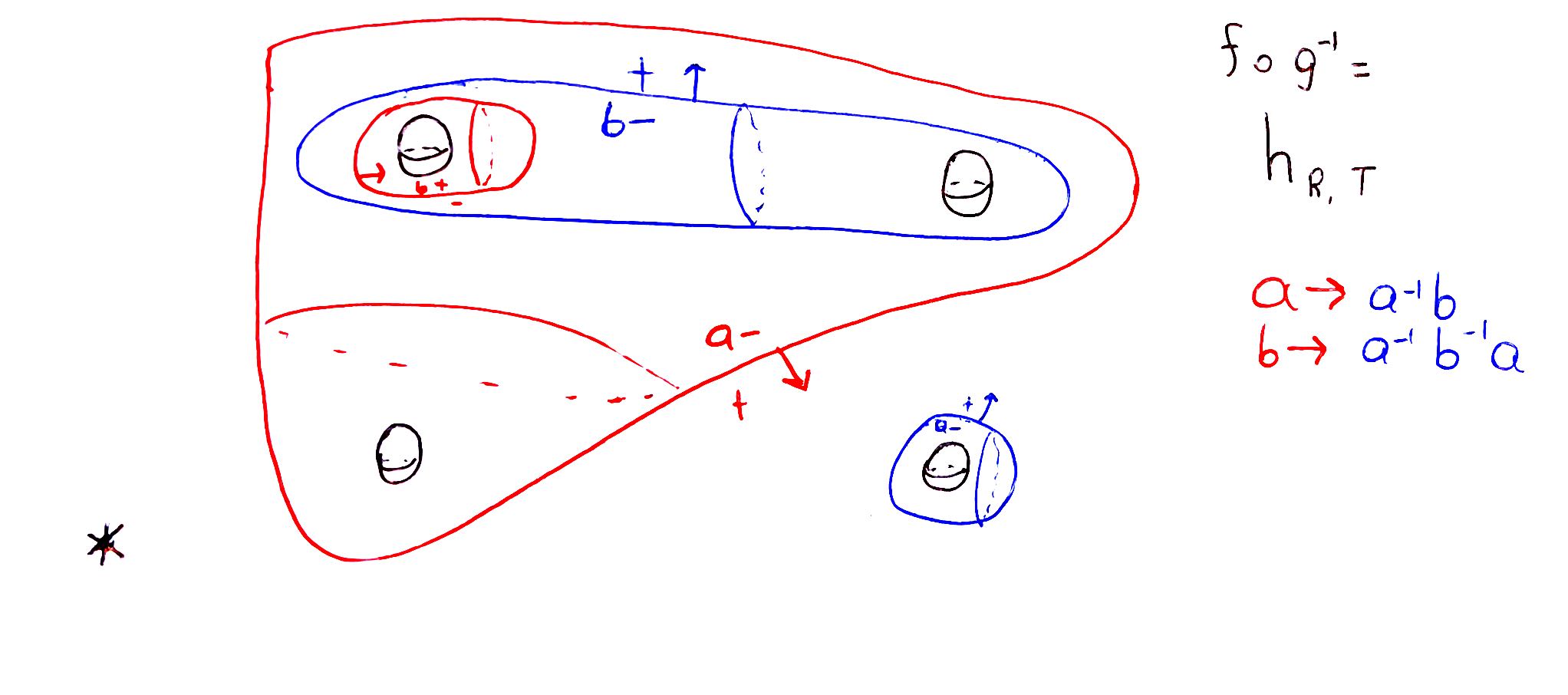}\\
\end{center}
    We draw our two homeomorphisms on top of each other and now we are able to trace our new automorphism $h = h_{\mathfrak{R},\mathfrak{T}}$. We notice that there are no intersections between the two sphere bases so deg$(h) = 0$.\\

\begin{center}
    \includegraphics[width=9.0cm]{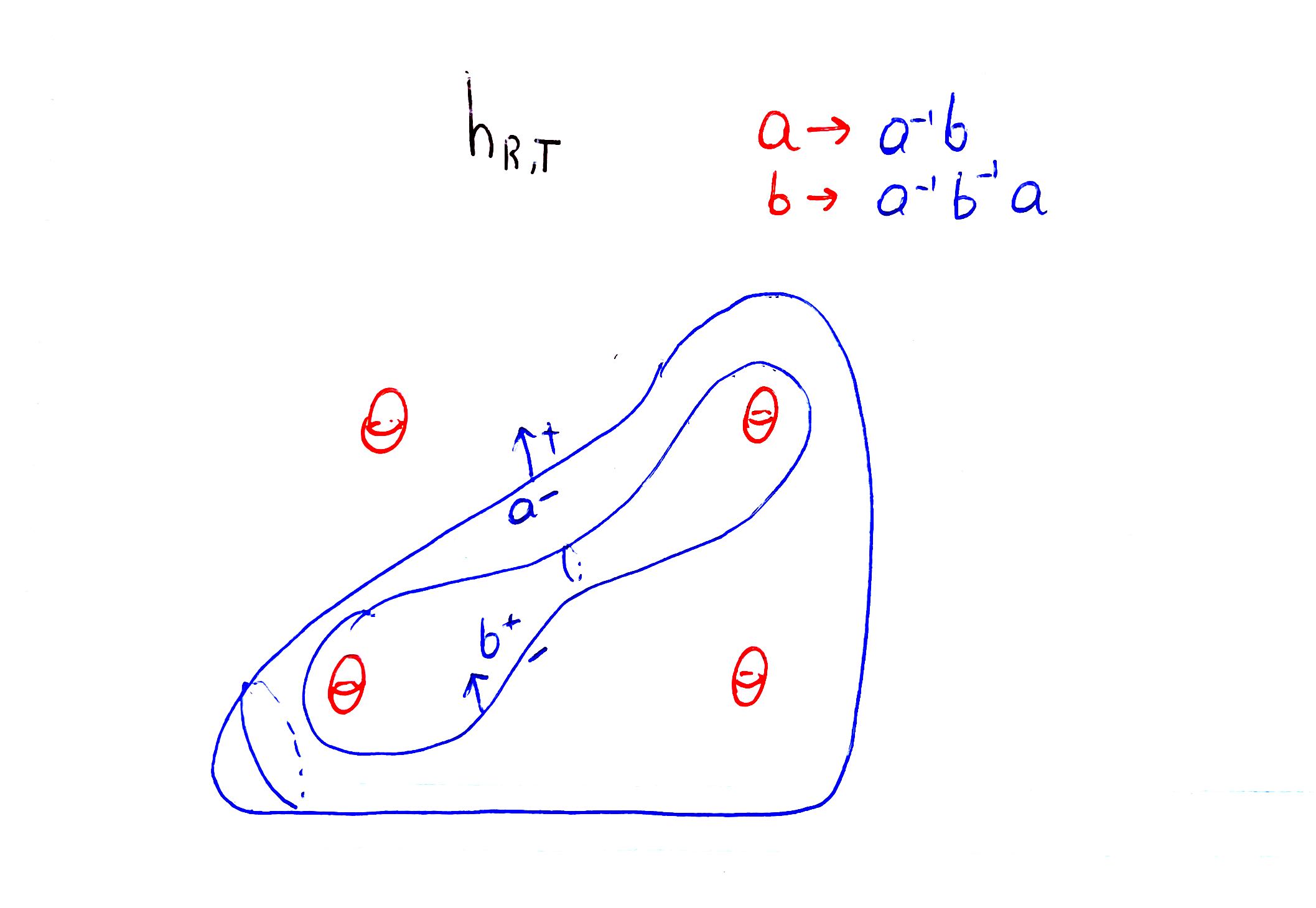}\\
    \end{center}
    Now that we have described our new automorphism $h$, we can draw a corresponding homeomorphism of $M_2$ that we will have that deg$(f \circ g) = 0.$\\

We can use this information to come up with the following conjecture about the degree of the composition of two automorphisms in $F_n$:

\begin{conjecture}{\label{Conj}}
    If $f$ and $g$ are automorphisms of $F_n$, then\\
    deg$(f \circ g) \leq (n + \text{deg}(f))(n+\text{deg}(g))$.
\end{conjecture}

\section{0 Degree Automorphisms}

\indent In his honor's thesis, Muller came up with a bound for the composition of two 0-degree automorphisms of $F_n$. His theorem is as follows:

\begin{theorem}[\cite{Muller2014}]\label{MullerThm}
    If $f$ and $g$ are 0-degree automorphisms of $F_n$, then the degree of $f \circ g$ is at most $n^2$.
\end{theorem}

This verifies Conjecture \ref{Conj} in the case where deg$(f) = $ deg$(g) = 0$.
Muller noted in his paper that although the bound for $f \circ g \leq
 4$, he states there is no pair of degree zero automorphisms of $F_2$ whose composition is degree four \cite{Muller2014}. We see an example below where we have deg$(f) = \text{deg}(g) = 0$ and when we take $f \circ g$, there are 4 intersections. 

\begin{center}
    \includegraphics[width=10.0cm]{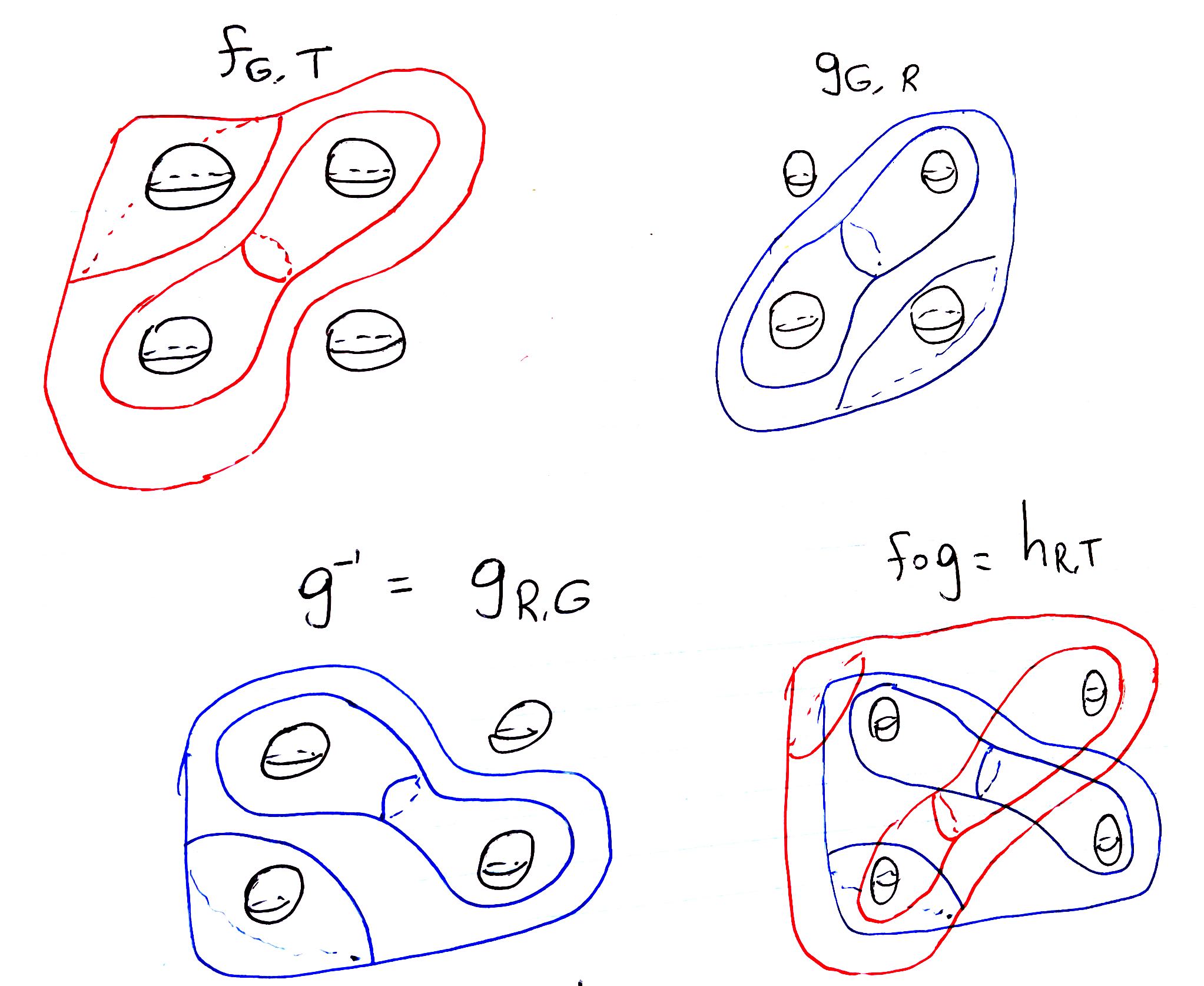}\\
\end{center}

It is not clear that Muller knew of this example but there are different possible labelings and orientations for this figure above. I worked through all of the possibilities and drew corresponding homeomorphisms using the methods above. Muller is correct is his statement that although the bound is 4,  there do not exist labeling and orientations for 4 intersections that lead to a degree 4 automorphism.

\section{Higher Degree Automorphisms}

In order for us to count the number of intersections, we will make use of the number of components each sphere has. The key result which allows us to attain our bound uses the following theorem that describes how we can remove intersections while retaining the same homeomorphism by way of isotopy. By working in the bicollar of the sphere, we are not changing anything topologically while still removing an intersection, helping us to find a bound for our degree.

\begin{theorem}{\label{IntThm}}
    Let $f_{\mathfrak{G}.\mathfrak{T}}$ be a homeomorphism of $M_n$ that maps $\mathfrak{T}$ to $\mathfrak{G}$ such that $\mathfrak{G} \cap \mathfrak{T} = \emptyset.$ and let $g_{\mathfrak{G}.\mathfrak{R}}$ be an homeomorphism of $F_n$ that maps $\mathfrak{R}$ to $\mathfrak{G}$ with deg$(g)=0$. Then, each component of each $\mathfrak{R}_i \in \mathfrak{R}$ will intersect with another component of each sphere $\mathfrak{T}_i \in \mathfrak{T}$ at most once in $\widetilde{M}_n$.
\end{theorem}

\begin{proof}
    Without loss of generality, let deg$(g)=0$. Since $\mathfrak{G} \cap \mathfrak{T} = \emptyset.$, each $\mathfrak{T}_i$ is a properly embedded sphere in $M_n$. Assume a component of $\mathfrak{R}_j$ intersects a sphere in $\mathfrak{T}$ more than once. For sake of argument, we will go through the proof of two intersections, but this process can be repeated recursively until there is one intersection remaining.

    The component of $\mathfrak{R}_j$ will also be properly embedded in $M_n$ and each intersection with $\mathfrak{T}$ will be a circle that does not contain anything else since $\mathfrak{T}_i$ is a sphere. We know that a circle bounds a disc on $\mathfrak{T}_i$ since our degree is 0, so this makes finding our bicollar easier.

    \begin{lemma}
        Working in this bicollar will not change the Euler Characteristic of the component we are working in since we are removing an annulus and 2 discs, but connecting the removed discs with an annulus and ``patching up'' our removed annulus with 2 discs.
    \end{lemma}

    Take one of the intersections between $\mathfrak{T}_i$ and the component of $\mathfrak{R}_j$ and collar the degree 0 sphere where the intersection occurs. We can then continue working in the collar of $\mathfrak{T}_i$ and move the $\mathfrak{R}_j$ component ``outside'' of the sphere. When we reattach this tube to itself where the other intersection occurs, we remain in the collar and have removed our initial intersection. In our illustrations below, we outline this procedure.

    Any curve that passed through our initial two intersections can now miss the sphere entirely and any curve that passed through one intersection only will continue to only pass through one intersection, although it may be different from the original.
\end{proof}

\begin{center}
\includegraphics[width=9cm]{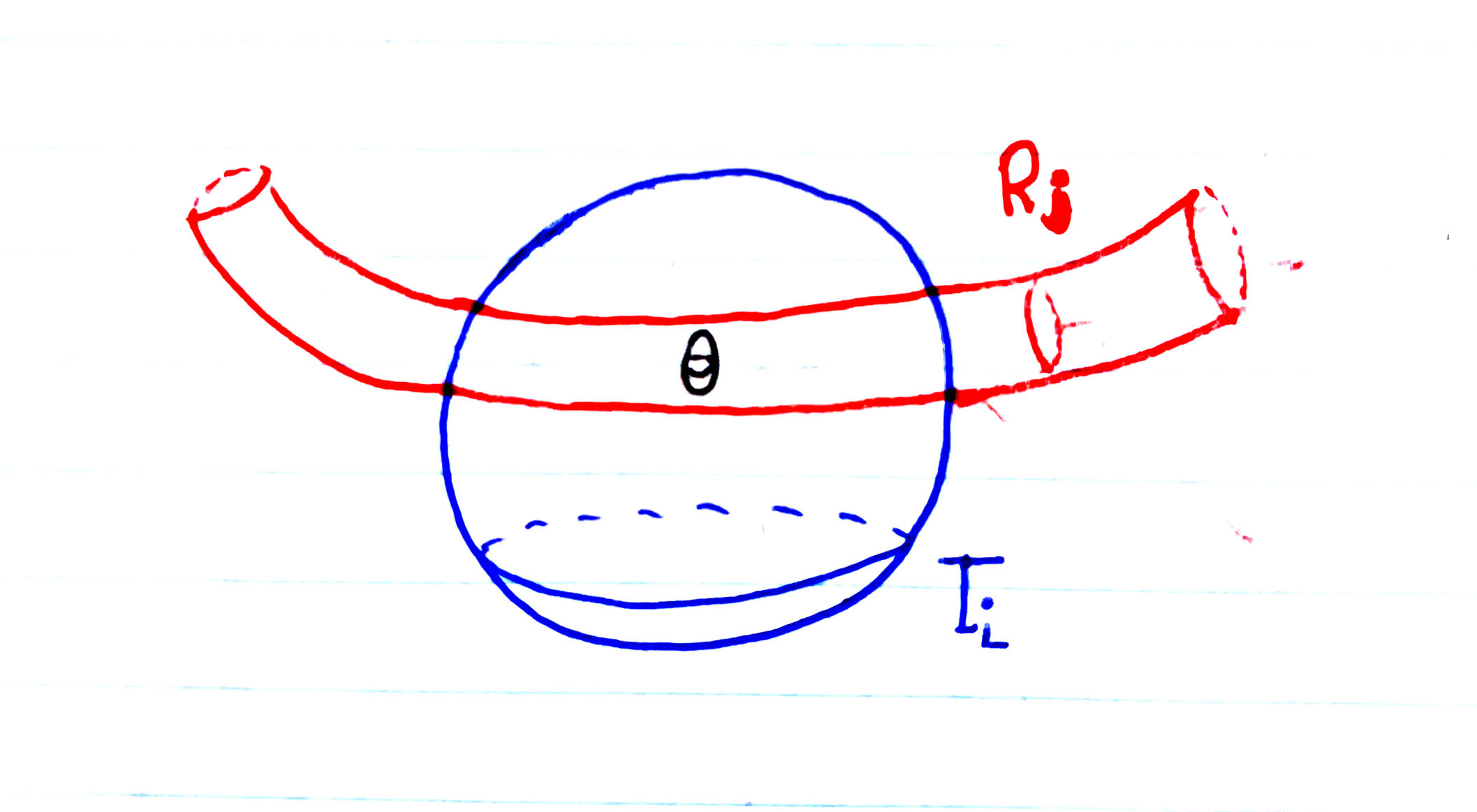}\\
\end{center}
Here we can see a component of $\mathfrak{R}$ intersecting a sphere in $\mathfrak{T}$ in two different spots in the $\widetilde{M}_n$ space. In our illustration, we have drawn the component of $\mathfrak{R}$ as a cylinder.
\begin{center}
    \includegraphics[width=9cm]{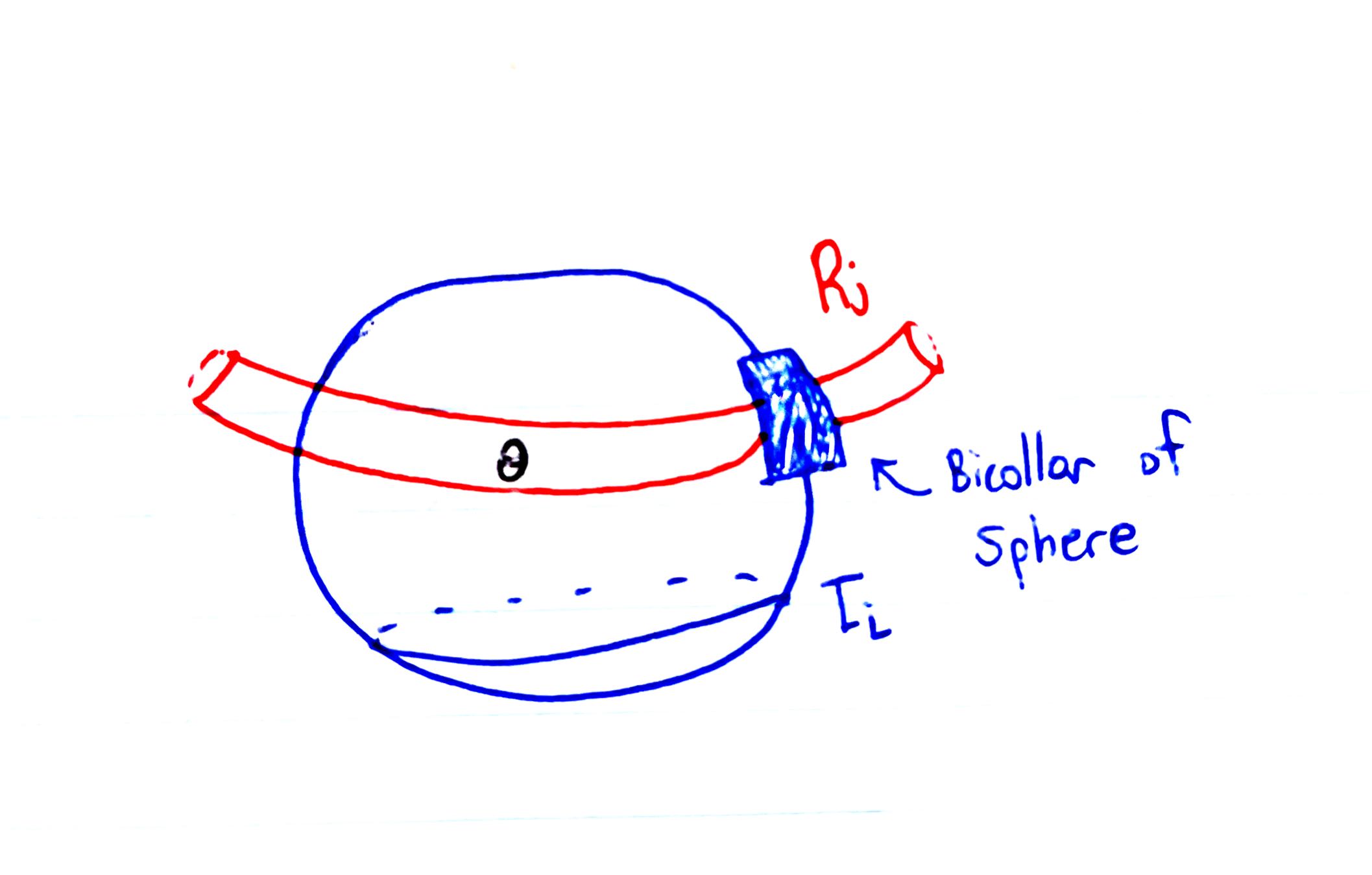}\\
\end{center}
In this figure we are working in the bicollar of the sphere to start working to remove the intersection on the right side.
\begin{center}
    \includegraphics[width=9cm]{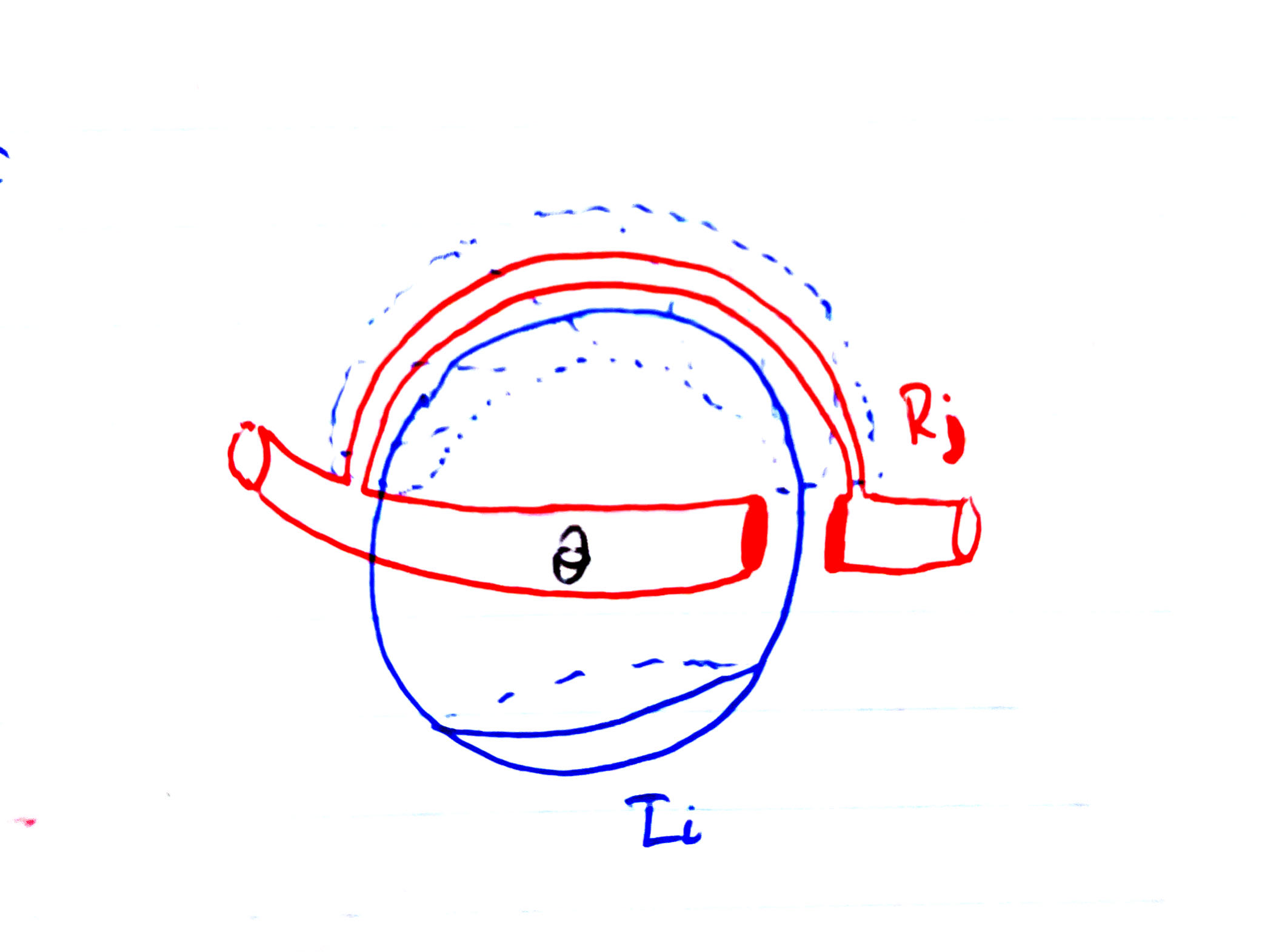}\\
\end{center}
We work around the entire bicollar of the sphere, wrapping around to the left intersection. Inside the collar, we remove the intersection on the right, tube it above the sphere, and reattach it on the outside of the left of the sphere, thus removing the intersection on the right.

\begin{theorem}
    If $f$ and $g$ are automorphisms of $F_n$ with deg$(f)=0$ and deg$(g) = k$, then the degree of $f \circ g$ is at most $n(k+n)$.
\end{theorem}

\begin{proof} 
    The case of deg$(g) = 0$ follows by Theorem \ref{MullerThm}\cite{Muller2014}.

    Since deg$(f) = 0$, we know that the total number of components that exist in $F_n$ is $n$ because no sphere will intersect with the standard sphere basis. Thus, we have $n$ spheres from $f_{\mathfrak{G}.\mathfrak{T}}$. Now we know that deg$(g) = k$ and by Theorem \ref{IntThm}, each component from a sphere in $g_{\mathfrak{G}.\mathfrak{R}}$ will intersect a sphere in $f_{\mathfrak{G}.\mathfrak{T}}$ at most once.
    
    At this point, our proof becomes a counting argument. We know that $g_{\mathfrak{G}.\mathfrak{R}}$ will consist of $n+k$ components that intersect the $n$ sphere of $f_{\mathfrak{G}.\mathfrak{T}}$ at most once. Hence, the total number of intersections that exist must be at most $n(n+k)$, and our proof is done.
    
\end{proof}

\section{Further Work}
    We have a conjecture for the bound for the number of intersections between two automorphisms of any degree. However, Muller proved this for when both are degree 0 and we have proven this for when one is degree 0. Further work can be done to prove our conjecture for all higher-degree automorphisms of $F_n$.

\section*{Acknowledgements}
I would like to thank Dr. Clifford for introducing me to the topic and giving me direction for the research. I would also like to thank all of the faculty in the Department of Mathematics and Statistics at TCNJ who have helped guide me along the right path and shown me what the intersection of research and teaching looks like. These lessons are ones that I will carry with me as I pursue a PhD in Mathematics at the University of Binghamton.

\newpage
\bibliographystyle{amsplain}
\bibliography{HonorsBib}

\end{document}